\documentclass[oneside,reqno,english]{amsart}
\usepackage[T1]{fontenc}
\usepackage[utf8]{inputenc}
\usepackage{xcolor}
\usepackage{babel}
\usepackage{prettyref}
\usepackage{amstext}
\usepackage{amsthm}
\usepackage{amssymb}
\usepackage[all]{xy}
\usepackage[pdfusetitle,
 bookmarks=true,bookmarksnumbered=false,bookmarksopen=false,
 breaklinks=false,pdfborder={0 0 0},pdfborderstyle={},backref=false,colorlinks=false]
 {hyperref}
\hypersetup{
 colorlinks=true,citecolor=blue,linkcolor=blue,linktocpage=true}

\makeatletter
\numberwithin{equation}{section}
\numberwithin{figure}{section}

\usepackage{prettyref}

\newrefformat{cor}{Corollary~\ref{#1}}
\newrefformat{subsec}{Section~\ref{#1}}
\newrefformat{lem}{Lemma~\ref{#1}}
\newrefformat{thm}{Theorem~\ref{#1}}
\newrefformat{sec}{Section~\ref{#1}}
\newrefformat{chap}{Chapter~\ref{#1}}
\newrefformat{prop}{Proposition~\ref{#1}}
\newrefformat{exa}{Example~\ref{#1}}
\newrefformat{tab}{Table~\ref{#1}}
\newrefformat{rem}{Remark~\ref{#1}}
\newrefformat{def}{Definition~\ref{#1}}
\newrefformat{fig}{Figure~\ref{#1}}
\newrefformat{claim}{Claim~\ref{#1}}
\newrefformat{assu}{Assumption~\ref{#1}}

\makeatother

\providecommand\theoremname{Theorem}
\theoremstyle{plain}
\newtheorem{thm}{\protect\theoremname}[section]
\providecommand\lemmaname{Lemma}
\newtheorem{lem}[thm]{\protect\lemmaname}
\providecommand\corollaryname{Corollary}
\newtheorem{cor}[thm]{\protect\corollaryname}
\begin{document}
\title{UCP Extension in Dimension Three}
\begin{abstract}
We show that every UCP extension problem for a finite dimensional
operator system can be reduced to three dimensions after finite matrix
amplification. The set of extensions is preserved by an affine bijection.
The matrix size is of order $\sqrt{d}$, where $d$ is the dimension
of the original system. This also gives a three dimensional (smallest)
counterexample to Arveson's hyperrigidity conjecture.
\end{abstract}

\author{James Tian}
\address{Mathematical Reviews, 535 W. William St, Suite 210, Ann Arbor, MI
48103, USA}
\email{james.ftian@gmail.com}
\maketitle

\section{Introduction}\label{sec:1}

Let $S\subset A=C^{*}\left(S\right)$ be an operator system and let
$\phi:S\longrightarrow B$ be UCP. The extension problem for $\phi$
is described by 
\[
\mathcal{E}_{A}\left(\phi,B\right)=\left\{ \Phi:A\longrightarrow B\mathrel{\big|}\Phi\text{ is UCP and }\Phi|_{S}=\phi\right\} .
\]
When $B=B\left(H\right)$ this set is nonempty by Arveson's extension
theorem \cite{MR0253059}. There is much more in the set than existence.
The unique extension property asks when the fiber above the restriction
of a representation consists of one map. Boundary representations,
the $C^{*}$-envelope, and hyperrigidity are built around this extension
and uniqueness theory \cite{MR0253059,MR394232,MR2425180,MR2823981}.

We ask how small an operator system carrying such an extension problem
can be. For finite dimensional systems, the answer is three dimensional.

Suppose that $S\subset A=C^{*}(S)$ has dimension $d$, and put 
\[
r=\max\left\{ 1,\left\lceil \sqrt{d-1}\right\rceil \right\} .
\]
We construct a three dimensional operator system 
\[
R\left(S\right)\subset M_{3r}\left(S\right)
\]
which generates $M_{3r}\left(A\right)$. (The ambient matrix size
grows on the order of $\sqrt{d}$, a reduction from the linear bounds
in recent general constructions \cite{HS2026}.) For every unital
$C^{*}$-algebra $B$ and every UCP map $\phi:S\longrightarrow B$,
there is a UCP map 
\[
\widehat{\phi}:R\left(S\right)\longrightarrow M_{3r}\left(B\right)
\]
for which amplification gives an affine bijection 
\[
\begin{aligned}\mathcal{E}_{A}\left(\phi,B\right) & \longrightarrow\mathcal{E}_{M_{3r}(A)}\left(\widehat{\phi},M_{3r}(B)\right),\\
\psi & \longmapsto id_{M_{3r}}\otimes\psi.
\end{aligned}
\]
The map $\phi$ is arbitrary. In particular, it need not be the restriction
of a representation.

Thus every UCP extension problem on a finite dimensional operator
system can be represented by a three dimensional operator system,
where the whole extension set is retained. If there are many extensions,
they remain. If there is one, there is one after the reduction.

The distinction is as follows. The ambient $C^{*}$-algebra becomes
a matrix algebra. We are not replacing $S$ by an isomorphic three
dimensional system. It is the system carrying the prescribed extension
data that has dimension three.

Since a three dimensional operator system has the form 
\[
span\left\{ I,T,T^{*}\right\} 
\]
for one operator $T$, there is another way to read the result. After
finite matrix amplification, the extension problem of an arbitrary
finite dimensional operator system can be stored in one operator.
The original problem has not become simple. Its complexity is now
inside the matrix entries of $T$. Small linear dimension says little
about the complexity of the UCP extensions.

Recent work on the Smith--Ward problem gives a rather different appearance
of the same low dimensional setting. Harris constructed four dimensional
operator systems without the lifting property \cite{HS2025}. Scherer
then constructed a three dimensional operator system giving a negative
solution to the Smith--Ward problem \cite{SM2026}. Harris subsequently
showed that related three dimensional constructions occur broadly
after finite matrix amplification \cite{HS2026}. Diagonal selfadjoint
operators and sparse block matrices occur naturally in these arguments.
So do Stinespring representations, norm propagation, and multiplicative
domains.

There is another consequence of the reduction. Scherer constructed
a finite dimensional counterexample to Arveson's hyperrigidity conjecture
\cite{MR4956742}. Every irreducible representation is a boundary
representation, but the system is not hyperrigid. Applying our reduction
gives a three dimensional system with the same two properties. This
is the smallest possible dimension for a counterexample. Indeed, a
two dimensional operator system for which every irreducible representation
is boundary must equal its generated $C^{*}$-algebra.

In the earlier constructions, the rigidity is pushed all the way to
uniqueness. We use this kind of local rigidity, but stop it at one
block. The rigidity ends there. The freedom that remains is exactly
the original extension set. That block contains the original operator
system. The other blocks make the scalar matrix algebra rigid. An
arbitrary UCP extension then has the form 
\[
\Psi=id_{M_{3r}}\otimes\psi.
\]
Nothing at this stage determines $\psi$ on all of $A$. Returning
to the untouched block gives 
\[
\psi|_{S}=\phi.
\]
The freedom in that block remains. The corner calculation identifies
it with the full set of extensions of $\phi$.

This differs from constructing a particular three dimensional system
with a lifting or rigidity property. We start with an arbitrary finite
dimensional $S$, an arbitrary unital target $B$, and an arbitrary
UCP map $\phi:S\longrightarrow B$. Its extension set is carried to
dimension three. We are not aware of an earlier result which replaces
an arbitrary UCP extension set by the ordinary extension set of a
three dimensional operator system in this way. Because the correspondence
is an affine bijection, nothing in the extension problem is collapsed.
The whole convex geometry of the extension set passes to dimension
three.

\subsection*{Literature review}

The extension problem belongs to a much larger theory of completely
positive maps and operator systems. Stinespring's representation theorem
and Choi's Schwarz inequality are basic here \cite{MR69403,MR0355615}.
The extension and order structure was developed further by Choi and
Effros, Wittstock, Hamana, and Ruan \cite{MR430809,MR609438,MR566081,MR929239}.
There are several directions from this point. Abstract operator systems,
duality, tensor products, and nuclearity have been studied in \cite{MR1600080,MR1929500,MR2607279,MR2793115,MR3173055,MR4351081}.
We use only a small part of this theory, mostly the elementary structure
of completely positive maps.

There is a separate, and closer, literature around the unique extension
property and boundary representations. Maximal dilations and pure
completely positive maps give two approaches to boundary representations
\cite{MR2132691,MR3173052,MR3430455}. Korovkin type questions lead
to the same uniqueness problem from another direction \cite{MR3421602}.
The unique extension property itself has also been studied directly
\cite{MR3864832,MR4654017}. Hyperrigidity grew out of this circle
of ideas, and there are by now a number of variants, criteria, and
counterexamples \cite{MR3331791,MR3850542,MR4227156}. Some more recent
developments are \cite{MR5009783,ClouatreThompson2025,MR5061752}.
Our question is narrower. We keep the whole fiber of extensions rather
than asking only whether that fiber consists of one map.

Matrix ranges give another part of the background. Their relation
with compact perturbations and low dimensional operator systems goes
back at least to Smith--Ward and the work around essential matrix
ranges \cite{MR595006,MR613797,MR677416}. Matrix convexity puts these
questions in a broader geometric setting \cite{MR1295129,MR1430718,MR1615970}.
Matrix states and their extreme points are closely connected \cite{MR1766112,MR2098611}.
There is a substantial later literature on dilations and matrix convex
sets \cite{MR3671511,MR3790504,MR3782992,MR4020034}. The recent work
of Farenick, Li, and Singla returns to essential matrix ranges and
the Smith--Ward problem from this side \cite{MR5118268}. Three dimensional
systems occur naturally throughout this literature, although for rather
different reasons than they do here.

\section{Reduction to three dimensions}
\begin{thm}
\label{thm:2-1} Let $S\subset A=C^{*}\left(S\right)$ be a finite
dimensional operator system, and let 
\[
d=\dim_{\mathbb{C}}S,\quad r=\max\left\{ 1,\left\lceil \sqrt{d-1}\right\rceil \right\} .
\]
There is a three dimensional operator system 
\[
R\left(S\right)\subset M_{3r}\left(S\right)
\]
such that 
\[
C^{*}\left(R\left(S\right)\right)=M_{3r}\left(A\right)
\]
and with the following property.

For every unital $C^{*}$-algebra $B$ and every UCP map 
\[
\phi:S\longrightarrow B,
\]
there is a UCP map 
\[
\widehat{\phi}:R\left(S\right)\longrightarrow M_{3r}\left(B\right)
\]
for which 
\[
\begin{aligned}\mathcal{E}_{A}\left(\phi,B\right) & \longrightarrow\mathcal{E}_{M_{3r}\left(A\right)}\left(\widehat{\phi},M_{3r}\left(B\right)\right),\\
\psi & \longmapsto id_{M_{3r}}\otimes\psi
\end{aligned}
\]
is an affine bijection.

Moreover, there is an operator $T\in M_{3r}\left(S\right)$ such that
\[
R\left(S\right)=span\left\{ I,T,T^{*}\right\} ,\quad C^{*}\left(T\right)=M_{3r}\left(A\right).
\]
\end{thm}

\begin{proof}
Choose selfadjoint elements $a_{1},\ldots,a_{d-1}\in S$ such that
\[
S=span\left\{ 1_{A},a_{1},\ldots,a_{d-1}\right\} .
\]
When $d=1$ there are no $a_{j}$.

An $r\times r$ selfadjoint matrix can store $r^{2}$ selfadjoint
elements. Since 
\[
r^{2}\geq d-1,
\]
we arrange $a_{1},\ldots,a_{d-1}$ in a selfadjoint matrix 
\[
X=\left[x_{pq}\right]^{r}_{p,q=1}\in M_{r}\left(S\right)
\]
as follows: 

Put elements on the diagonal one at a time. Then, put two unused elements
$a_{j},a_{k}$ in an off-diagonal pair by setting 
\[
x_{pq}=a_{j}+ia_{k},\quad x_{qp}=a_{j}-ia_{k}.
\]
If there's only one left, then let 
\[
x_{pq}=x_{qp}=a_{j}.
\]
Do this until we have used all the $a_{j}$'s, then set all remaining
entries of $X$ equal to zero. If $d=1$, take $X=0$.

Note every $a_{j}$ is a scalar linear combination of the entries
of $X$, so 
\[
S=span\left\{ 1_{A},x_{pq}:1\leq p,q\leq r\right\} .
\]

Let $Q,V\in M_{r}$, where 
\[
Q=E_{11}=\begin{pmatrix}1 & 0 & \cdots & \cdots & 0\\
0 & 0 & \cdots & \cdots & 0\\
\vdots & \vdots & \vdots & \vdots & \vdots\\
\vdots & \vdots & \vdots & \vdots & \vdots\\
0 & 0 & \cdots & \cdots & 0
\end{pmatrix},\quad V=\begin{pmatrix}0 & 0 & \cdots & 0 & 1\\
1 & 0 & \cdots & 0 & 0\\
0 & 1 & \ddots & 0 & 0\\
\vdots &  & \ddots &  & \vdots\\
0 & 0 & \cdots & 1 & 0
\end{pmatrix}.
\]
For $r=1$, let $V=1$.

We shall use the elementary identity 
\[
C^{*}\left(Q,V\right)=M_{r}.
\]
Indeed, 
\[
V^{i-1}Q=E_{i1}.
\]
Taking adjoints gives $E_{1j}$, and so 
\[
E_{i1}E_{1j}=E_{ij}.
\]
Thus $Q$ and $V$ generate all the matrix units of $M_{r}$.

We also use the identification 
\[
M_{3r}\left(A\right)\simeq M_{3}\left(M_{r}\left(A\right)\right).
\]
Define 
\[
D=\begin{pmatrix}I_{r} & 0 & 0\\
0 & 2I_{r} & 0\\
0 & 0 & 3I_{r}
\end{pmatrix},\quad K=\begin{pmatrix}Q & I_{r} & I_{r}\\
I_{r} & X & V\\
I_{r} & V^{*} & 0
\end{pmatrix};
\]
both are selfadjoint. Set 
\[
R\left(S\right)=span\left\{ I_{3r},D,K\right\} .
\]
Since all entries of $D$ and $K$ belong to $S$, 
\[
R\left(S\right)\subset M_{3r}\left(S\right).
\]

The three generators are linearly independent, since 
\begin{gather*}
\alpha I+\beta D+\gamma K=\begin{pmatrix}\left(\alpha+\beta\right)I_{r}+\gamma Q & \gamma I_{r} & \gamma I_{r}\\
\gamma I_{r} & \left(\alpha+2\beta\right)I_{r}+\gamma X & \gamma V\\
\gamma I_{r} & \gamma V^{*} & \left(\alpha+3\beta\right)I_{r}
\end{pmatrix}=0\\
\Updownarrow\\
\alpha=\beta=\gamma=0.
\end{gather*}
Thus 
\[
\dim_{\mathbb{C}}R\left(S\right)=3.
\]

We will need a few properties of this construction. 
\end{proof}

\begin{lem}
\label{lem:2-2} $R\left(S\right)$ generates the full matrix algebra
over $A$, i.e., 
\begin{equation}
C^{*}\left(D,K\right)=M_{3r}\left(A\right).\label{eq:2-1}
\end{equation}
\end{lem}

\begin{proof}
Let $\mathcal{C}=C^{*}\left(D,K\right)$. The three projections $P_{j}$
below are in $\mathcal{C}$, where 
\begin{align}
P_{1} & =\frac{1}{2}\left(D-2I\right)\left(D-3I\right)=\begin{pmatrix}I_{r} & 0 & 0\\
0 & 0 & 0\\
0 & 0 & 0
\end{pmatrix},\label{eq:2-2}\\
P_{2} & =-\left(D-I\right)\left(D-3I\right)=\begin{pmatrix}0 & 0 & 0\\
0 & I_{r} & 0\\
0 & 0 & 0
\end{pmatrix},\label{eq:2-3}\\
P_{3} & =\frac{1}{2}\left(D-I\right)\left(D-2I\right)=\begin{pmatrix}0 & 0 & 0\\
0 & 0 & 0\\
0 & 0 & I_{r}
\end{pmatrix}.\label{eq:2-4}
\end{align}

We have 
\[
P_{1}KP_{2}=\begin{pmatrix}0 & I_{r} & 0\\
0 & 0 & 0\\
0 & 0 & 0
\end{pmatrix},\quad P_{1}KP_{3}=\begin{pmatrix}0 & 0 & I_{r}\\
0 & 0 & 0\\
0 & 0 & 0
\end{pmatrix}.
\]
Taking adjoints gives 
\[
\begin{pmatrix}0 & 0 & 0\\
I_{r} & 0 & 0\\
0 & 0 & 0
\end{pmatrix},\quad\begin{pmatrix}0 & 0 & 0\\
0 & 0 & 0\\
I_{r} & 0 & 0
\end{pmatrix}.
\]
Also, 
\[
\begin{pmatrix}0 & 0 & 0\\
I_{r} & 0 & 0\\
0 & 0 & 0
\end{pmatrix}\begin{pmatrix}0 & 0 & I_{r}\\
0 & 0 & 0\\
0 & 0 & 0
\end{pmatrix}=\begin{pmatrix}0 & 0 & 0\\
0 & 0 & I_{r}\\
0 & 0 & 0
\end{pmatrix},
\]
and taking adjoints gives the remaining off-diagonal block. Thus $\mathcal{C}$
contains every $3\times3$ block matrix having $I_{r}$ in one block
and zero elsewhere.

Next, 
\[
P_{1}KP_{1}=\begin{pmatrix}Q & 0 & 0\\
0 & 0 & 0\\
0 & 0 & 0
\end{pmatrix},\quad P_{2}KP_{3}=\begin{pmatrix}0 & 0 & 0\\
0 & 0 & V\\
0 & 0 & 0
\end{pmatrix}.
\]
Hence 
\[
\begin{pmatrix}0 & I_{r} & 0\\
0 & 0 & 0\\
0 & 0 & 0
\end{pmatrix}\begin{pmatrix}0 & 0 & 0\\
0 & 0 & V\\
0 & 0 & 0
\end{pmatrix}\begin{pmatrix}0 & 0 & 0\\
0 & 0 & 0\\
I_{r} & 0 & 0
\end{pmatrix}=\begin{pmatrix}V & 0 & 0\\
0 & 0 & 0\\
0 & 0 & 0
\end{pmatrix}.
\]
Thus the upper-left block contains both $Q$ and $V$. Since 
\[
C^{*}\left(Q,V\right)=M_{r},
\]
we have 
\[
\begin{pmatrix}E_{pq} & 0 & 0\\
0 & 0 & 0\\
0 & 0 & 0
\end{pmatrix}\in\mathcal{C}
\]
for every $p,q$.

Using the block matrices already obtained, these matrix units can
be moved to any of the nine blocks. For example, 
\[
\begin{pmatrix}0 & 0 & 0\\
I_{r} & 0 & 0\\
0 & 0 & 0
\end{pmatrix}\begin{pmatrix}E_{pq} & 0 & 0\\
0 & 0 & 0\\
0 & 0 & 0
\end{pmatrix}\begin{pmatrix}0 & I_{r} & 0\\
0 & 0 & 0\\
0 & 0 & 0
\end{pmatrix}=\begin{pmatrix}0 & 0 & 0\\
0 & E_{pq} & 0\\
0 & 0 & 0
\end{pmatrix}.
\]
Thus $\mathcal{C}$ contains every scalar matrix unit of $M_{3r}$.

Finally, 
\[
P_{2}KP_{2}=\begin{pmatrix}0 & 0 & 0\\
0 & X & 0\\
0 & 0 & 0
\end{pmatrix}.
\]
Write 
\[
X=\begin{pmatrix}x_{11} & \cdots & x_{1r}\\
\vdots &  & \vdots\\
x_{r1} & \cdots & x_{rr}
\end{pmatrix}.
\]
Since all scalar matrix units have already been recovered, 
\[
\begin{pmatrix}0 & E_{1p} & 0\\
0 & 0 & 0\\
0 & 0 & 0
\end{pmatrix}\begin{pmatrix}0 & 0 & 0\\
0 & X & 0\\
0 & 0 & 0
\end{pmatrix}\begin{pmatrix}0 & 0 & 0\\
E_{q1} & 0 & 0\\
0 & 0 & 0
\end{pmatrix}=\begin{pmatrix}E_{1p}XE_{q1} & 0 & 0\\
0 & 0 & 0\\
0 & 0 & 0
\end{pmatrix}.
\]
But 
\[
E_{1p}XE_{q1}=\begin{pmatrix}x_{pq} & 0 & \cdots & 0\\
0 & 0 & \cdots & 0\\
\vdots & \vdots &  & \vdots\\
0 & 0 & \cdots & 0
\end{pmatrix}.
\]
Thus $\mathcal{C}$ contains a matrix having $x_{pq}$ in one scalar
position and zero everywhere else, for every $p,q$.

The same is true with $1_{A}$ in that position. Since 
\[
S=span\left\{ 1_{A},x_{pq}:1\leq p,q\leq r\right\} 
\]
and 
\[
A=C^{*}\left(S\right),
\]
that corner contains every element of $A$. The scalar matrix units
then move an arbitrary element of $A$ to any matrix position. Hence
\[
M_{3r}\left(A\right)\subset\mathcal{C}.
\]
The other direction $\mathcal{C}\subset M_{3r}\left(A\right)$ is
immediate, so \prettyref{eq:2-1} holds.
\end{proof}

\begin{lem}
\label{lem:2-3} Let $B$ be a unital $C^{*}$-algebra and let $\phi:S\longrightarrow B$
be UCP. Set 
\[
X_{\phi}=\phi^{\left(r\right)}\left(X\right)=\left[\phi\left(x_{pq}\right)\right]^{r}_{p,q=1},\quad\widehat{\phi}=\phi^{\left(3r\right)}|_{R\left(S\right)}
\]
and 
\[
K_{\phi}=\phi^{\left(3r\right)}\left(K\right)=\begin{pmatrix}Q & I_{r} & I_{r}\\
I_{r} & X_{\phi} & V\\
I_{r} & V^{*} & 0
\end{pmatrix}\in M_{3r}\left(B\right).
\]

If $\Psi:M_{3r}\left(A\right)\longrightarrow M_{3r}\left(B\right)$
is a UCP extension of $\widehat{\phi}$, i.e., 
\[
\Psi|_{R\left(S\right)}=\widehat{\phi},
\]
then 
\[
M_{3r}\otimes1_{A}\subset MD\left(\Psi\right)
\]
and 
\[
\Psi\left(z\otimes1_{A}\right)=z\otimes1_{B}
\]
for every $z\in M_{3r}$. 

In other words, any UCP extension of $\widehat{\phi}$ must fix the
scalar matrix algebra and behave multiplicatively with respect to
it.
\end{lem}

\begin{proof}
We use the same notation for a scalar matrix over $A$ and the corresponding
scalar matrix over $B$. Since 
\[
\widehat{\phi}\left(D\right)=D,\quad\widehat{\phi}\left(K\right)=K_{\phi},
\]
we have 
\[
\Psi\left(D\right)=D,\quad\Psi\left(K\right)=K_{\phi}.
\]

Choose a faithful representation of $B$ on a Hilbert space $H_{0}$,
and identify $B$ with its image. We then regard $M_{3r}\left(B\right)$
as acting on 
\[
H=H^{r}_{0}\oplus H^{r}_{0}\oplus H^{r}_{0}.
\]
By Stinespring's theorem, there exist a Hilbert space $L$, a unital
$*$-representation 
\[
\rho:M_{3r}\left(A\right)\longrightarrow B\left(L\right),
\]
and an isometry $W:H\longrightarrow L$ such that 
\[
\Psi\left(y\right)=W^{*}\rho\left(y\right)W,\quad y\in M_{3r}\left(A\right).
\]

Let $P_{1},P_{2},P_{3}$ be the projections from \prettyref{eq:2-2}--\prettyref{eq:2-4},
i.e., 
\[
P_{1}=E_{11}\otimes I_{r},\quad P_{2}=E_{22}\otimes I_{r},\quad P_{3}=E_{33}\otimes I_{r}.
\]
We use the same notation for the corresponding projections on $H$.

\textbf{1.} We first show that, for $j=1,2,3$, 
\begin{equation}
\rho\left(P_{j}\right)W=WP_{j},\qquad\vcenter{\xymatrix{H\ar[r]^{W}\ar[d]_{P_{j}} & L\ar[d]^{\rho\left(P_{j}\right)}\\
H\ar[r]_{W} & L
}
}\label{eq:2-5}
\end{equation}

Let $\eta\in P_{1}H$. Then $D\eta=\eta$, and since $\Psi\left(D\right)=D$,
\[
\begin{aligned}0 & =\left\langle \eta,\left(\Psi\left(D\right)-I\right)\eta\right\rangle =\left\langle W\eta,\left(\rho\left(D\right)-I\right)W\eta\right\rangle .\end{aligned}
\]
Now 
\[
\rho\left(D\right)-I=\rho\left(D-I\right)\geq0\Longrightarrow\left(\rho\left(D\right)-I\right)W\eta=0,
\]
so 
\[
\rho\left(D\right)W\eta=W\eta.
\]
Thus $W\eta$ belongs to the $1$-eigenspace of $\rho\left(D\right)$.
Since 
\[
\rho\left(D\right)=\rho\left(P_{1}\right)+2\rho\left(P_{2}\right)+3\rho\left(P_{3}\right),
\]
this eigenspace is $\rho\left(P_{1}\right)L$. Therefore 
\[
W\left(P_{1}H\right)\subset\rho\left(P_{1}\right)L.
\]
The same argument with $3I-\rho\left(D\right)$ gives 
\[
W\left(P_{3}H\right)\subset\rho\left(P_{3}\right)L.
\]

For $P_{2}$ we need $K$. Write 
\[
KP_{3}=K_{\phi}P_{3}=\begin{pmatrix}0 & 0 & I_{r}\\
0 & 0 & V\\
0 & 0 & 0
\end{pmatrix}=\underset{U}{\underbrace{\begin{pmatrix}0 & 0 & I_{r}\\
0 & 0 & 0\\
0 & 0 & 0
\end{pmatrix}}}+\underset{Z}{\underbrace{\begin{pmatrix}0 & 0 & 0\\
0 & 0 & V\\
0 & 0 & 0
\end{pmatrix}}}
\]
where 
\begin{equation}
U=E_{13}\otimes I_{r},\quad Z=E_{23}\otimes V.\label{eq:b-6}
\end{equation}
Both $U$ and $Z$ are partial isometries, with 
\begin{align*}
U^{*}U & =P_{3},\quad UU^{*}=P_{1}\\
Z^{*}Z & =P_{3},\quad ZZ^{*}=P_{2}
\end{align*}

Let $\eta\in P_{3}H$. Since 
\[
W\eta\in\rho\left(P_{3}\right)L,
\]
we have 
\begin{align}
\rho\left(K\right)W\eta & =\rho\left(K\right)\rho\left(P_{3}\right)W\eta\nonumber \\
 & =\rho\left(KP_{3}\right)W\eta=\rho\left(U\right)W\eta+\rho\left(Z\right)W\eta.\label{eq:2-6}
\end{align}
The two terms are orthogonal and each has norm $\left\Vert \eta\right\Vert $,
so 
\[
\left\Vert \rho\left(K\right)W\eta\right\Vert ^{2}=2\left\Vert \eta\right\Vert ^{2}.
\]
Similarly, 
\begin{equation}
K_{\phi}\eta=U\eta+Z\eta,\label{eq:2-7}
\end{equation}
and 
\[
\left\Vert K_{\phi}\eta\right\Vert ^{2}=2\left\Vert \eta\right\Vert ^{2}.
\]

Using 
\[
W^{*}\rho\left(K\right)W\eta=\Psi\left(K\right)\eta=K_{\phi}\eta
\]
it follows that 
\[
\left\Vert W^{*}\rho\left(K\right)W\eta\right\Vert =\left\Vert \rho\left(K\right)W\eta\right\Vert .
\]
Since $W$ is an isometry, $WW^{*}$ is the projection onto $W(H)$,
so we get 
\[
\rho\left(K\right)W\eta\in W\left(H\right).
\]
Therefore, 
\[
\rho\left(K\right)W\eta=W\left(W^{*}\rho\left(K\right)W\right)\eta=W\Psi\left(K\right)\eta=WK_{\phi}\eta.
\]
That is, 
\begin{equation}
\rho\left(K\right)W\eta=WK_{\phi}\eta.\label{eq:2-8}
\end{equation}

Combining \prettyref{eq:2-6}, \prettyref{eq:2-7} and \prettyref{eq:2-8},
we get 
\begin{equation}
\rho\left(U\right)W\eta+\rho\left(Z\right)W\eta=WU\eta+WZ\eta.\label{eq:2-9}
\end{equation}

Apply $\rho\left(P_{2}\right)$ to both sides of \prettyref{eq:2-9}.
Using \prettyref{eq:b-6}, 
\[
\rho\left(P_{2}\right)\left(\rho\left(U\right)W\eta+\rho\left(Z\right)W\eta\right)=\rho\left(Z\right)W\eta.
\]
On the other hand, since $U\eta\in P_{1}H$ and 
\[
W\left(P_{1}H\right)\subset\rho\left(P_{1}\right)L,
\]
we get 
\[
\rho\left(P_{2}\right)\left(WU\eta+WZ\eta\right)=\underset{=0}{\underbrace{\rho\left(P_{2}\right)WU\eta}}+\rho\left(P_{2}\right)WZ\eta.
\]
This gives 
\[
\rho\left(Z\right)W\eta=\rho\left(P_{2}\right)WZ\eta.
\]

Note that 
\begin{align*}
\left\Vert \rho\left(Z\right)W\eta\right\Vert ^{2} & =\left\langle \eta,W^{*}\rho\left(Z^{*}Z\right)W\eta\right\rangle \\
 & =\left\langle \eta,W^{*}\rho\left(P_{3}\right)W\eta\right\rangle =\left\langle \eta,W^{*}W\eta\right\rangle =\left\Vert \eta\right\Vert ^{2}
\end{align*}
and 
\[
\left\Vert WZ\eta\right\Vert =\left\Vert Z\eta\right\Vert =\left\Vert \eta\right\Vert .
\]
Therefore, 
\[
\left\Vert \rho\left(P_{2}\right)WZ\eta\right\Vert =\left\Vert WZ\eta\right\Vert .
\]

Since $\rho\left(P_{2}\right)$ is a projection, it follows that
\[
WZ\eta\in\rho\left(P_{2}\right)L.
\]
The operator $Z$ maps $P_{3}H$ onto $P_{2}H$. Therefore 
\[
W\left(P_{2}H\right)\subset\rho\left(P_{2}\right)L.
\]

We have shown 
\[
W\left(P_{j}H\right)\subset\rho\left(P_{j}\right)L,\quad j=1,2,3.
\]
Now $\sum P_{j}=I$ gives 
\[
\rho\left(P_{j}\right)W=\rho\left(P_{j}\right)\left(\sum WP_{k}\right)=WP_{j},
\]
which is \prettyref{eq:2-5}. 

Consequently, 
\[
\Psi\left(P_{j}\right)=W^{*}\rho\left(P_{j}\right)W=W^{*}WP_{j}=P_{j}.
\]
Since $P_{j}$ and $\Psi\left(P_{j}\right)$ are projections, 
\[
P_{j}\in MD\left(\Psi\right),\quad j=1,2,3.
\]

\textbf{2.} Consider now 
\[
V_{12}=\begin{pmatrix}0 & I_{r} & 0\\
0 & 0 & 0\\
0 & 0 & 0
\end{pmatrix},\quad V_{13}=\begin{pmatrix}0 & 0 & I_{r}\\
0 & 0 & 0\\
0 & 0 & 0
\end{pmatrix},\quad V_{23}=\begin{pmatrix}0 & 0 & 0\\
0 & 0 & V\\
0 & 0 & 0
\end{pmatrix}.
\]
Since all $P_{j}\in MD\left(\Psi\right)$ and $\Psi\left(P_{j}\right)=P_{j}$,
\[
\Psi\left(V_{12}\right)=\Psi\left(P_{1}KP_{2}\right)=P_{1}\Psi\left(K\right)P_{2}=V_{12}.
\]
Similarly, 
\[
\Psi\left(V_{13}\right)=V_{13},\quad\Psi\left(V_{23}\right)=V_{23}.
\]

Each $V_{ij}$ is a partial isometry satisfying 
\[
V^{*}_{ij}V_{ij}=P_{j},\quad V_{ij}V^{*}_{ij}=P_{i}.
\]
Since $\Psi(V_{ij})=V_{ij}$, 
\[
\Psi\left(V^{*}_{ij}V_{ij}\right)=\Psi\left(P_{j}\right)=P_{j}=V^{*}_{ij}V_{ij}=\Psi\left(V_{ij}\right)^{*}\Psi\left(V_{ij}\right),
\]
and similarly, 
\[
\Psi\left(V_{ij}V^{*}_{ij}\right)=\Psi\left(V_{ij}\right)\Psi\left(V_{ij}\right)^{*}.
\]
Hence 
\[
V_{12},V_{13},V_{23}\in MD\left(\Psi\right).
\]
It follows that
\[
E_{ij}\otimes I_{r}\in MD\left(\Psi\right),\quad\Psi\left(E_{ij}\otimes I_{r}\right)=E_{ij}\otimes I_{r},\quad1\leq i,j\leq3.
\]

Now put 
\[
Q_{1}=P_{1}KP_{1}=\begin{pmatrix}Q & 0 & 0\\
0 & 0 & 0\\
0 & 0 & 0
\end{pmatrix}.
\]
Since $P_{1}\in MD\left(\Psi\right)$, 
\[
\Psi\left(Q_{1}\right)=\Psi\left(P_{1}KP_{1}\right)=P_{1}\Psi\left(K\right)P_{1}=Q_{1}.
\]
Both $Q_{1}$ and $\Psi\left(Q_{1}\right)$ are projections, so 
\[
Q_{1}\in MD\left(\Psi\right).
\]

Since $MD\left(\Psi\right)$ is a $C^{*}$-algebra, 
\[
V_{12}V_{23}V^{*}_{13}=\begin{pmatrix}V & 0 & 0\\
0 & 0 & 0\\
0 & 0 & 0
\end{pmatrix}\in MD\left(\Psi\right).
\]
Thus the upper-left block contains both $Q$ and $V$ inside $MD\left(\Psi\right)$.
Since 
\[
C^{*}\left(Q,V\right)=M_{r},
\]
all scalar $r\times r$ matrices in that block belong to $MD\left(\Psi\right)$,
i.e., 
\[
E_{11}\otimes a\in MD(\Psi),\quad a\in M_{r}.
\]

Since $E_{ij}\otimes I_{r}\in MD(\Psi)$, we have 
\[
(E_{i1}\otimes I_{r})(E_{11}\otimes a)(E_{1j}\otimes I_{r})=E_{ij}\otimes a\in MD(\Psi)
\]
for every $a\in M_{r}$. Hence 
\[
M_{3r}\otimes1_{A}\subset MD(\Psi).
\]

Moreover, $\Psi$ fixes the generators of this scalar algebra. Its
restriction to the multiplicative domain is multiplicative, so 
\[
\Psi\left(z\otimes1_{A}\right)=z\otimes1_{B}
\]
for every $z\in M_{3r}$.
\end{proof}

\begin{lem}
\label{lem:2-4} Let $B$ be a unital $C^{*}$-algebra and let 
\[
\Psi:M_{3r}\left(A\right)\longrightarrow M_{3r}\left(B\right)
\]
be UCP. Suppose 
\[
M_{3r}\otimes1_{A}\subset MD\left(\Psi\right)
\]
and 
\[
\Psi\left(z\otimes1_{A}\right)=z\otimes1_{B}
\]
for every $z\in M_{3r}$. Then there is a unique UCP map 
\[
\psi:A\longrightarrow B
\]
such that 
\[
\Psi=id_{M_{3r}}\otimes\psi.
\]
\end{lem}

\begin{proof}
For $a\in A$, since $E_{11}\otimes1_{A}\in MD(\Psi)$, 
\[
\begin{aligned}\Psi\left(E_{11}\otimes a\right) & =\Psi\left(\left(E_{11}\otimes1_{A}\right)\left(E_{11}\otimes a\right)\left(E_{11}\otimes1_{A}\right)\right)\\
 & =\left(E_{11}\otimes1_{B}\right)\Psi\left(E_{11}\otimes a\right)\left(E_{11}\otimes1_{B}\right)=\begin{pmatrix}* & 0 & \cdots\\
0 & 0 & \cdots\\
\vdots & \vdots
\end{pmatrix}.
\end{aligned}
\]
Thus there is a unique linear map 
\[
\psi:A\longrightarrow B
\]
such that 
\[
\Psi\left(E_{11}\otimes a\right)=E_{11}\otimes\psi\left(a\right).
\]
Since $\Psi\left(E_{11}\otimes1_{A}\right)=E_{11}\otimes1_{B}$, the
map $\psi$ is unital.

For $1\leq i,j\leq3r$, 
\[
E_{ij}\otimes a=\left(E_{i1}\otimes1_{A}\right)\left(E_{11}\otimes a\right)\left(E_{1j}\otimes1_{A}\right).
\]
A similar MD argument gives 
\[
\Psi\left(E_{ij}\otimes a\right)=E_{ij}\otimes\psi\left(a\right).
\]
Hence $\Psi([a_{ij}])=[\psi(a_{ij})]$, i.e., 
\[
\Psi=id_{M_{3r}}\otimes\psi.
\]

Since $\Psi$ is completely positive, for every $k$ and every $\left[a_{ij}\right]\geq0$
in $M_{k}\left(A\right)$, 
\[
\left[\Psi\left(E_{11}\otimes a_{ij}\right)\right]\geq0.
\]
Thus 
\[
\left[E_{11}\otimes\psi\left(a_{ij}\right)\right]\geq0,
\]
and therefore 
\[
\left[\psi\left(a_{ij}\right)\right]\geq0.
\]
Hence $\psi$ is completely positive. The $11$ corner gives uniqueness. 
\end{proof}

\begin{proof}[Proof of \prettyref{thm:2-1}, continued]
By \prettyref{lem:2-2}, 
\[
C^{*}\left(R\left(S\right)\right)=M_{3r}\left(A\right).
\]
Now let $B$ be a unital $C^{*}$-algebra and let 
\[
\phi:S\longrightarrow B
\]
be UCP. Use the notation $X_{\phi}$, $K_{\phi}$, and $\widehat{\phi}$
from \prettyref{lem:2-3}. Then $\widehat{\phi}$ is UCP.

Let 
\[
\Psi:M_{3r}\left(A\right)\longrightarrow M_{3r}\left(B\right)
\]
be a UCP extension of $\widehat{\phi}$. By \prettyref{lem:2-3},
\[
M_{3r}\otimes1_{A}\subset MD\left(\Psi\right),\quad\Psi\left(z\otimes1_{A}\right)=z\otimes1_{B},\:\forall z\in M_{3r}.
\]
Then \prettyref{lem:2-4} gives a unique UCP map 
\[
\psi:A\longrightarrow B,\quad\Psi=id_{M_{3r}}\otimes\psi.
\]

Since $\Psi$ extends $\widehat{\phi}$, we have $\Psi\left(K\right)=K_{\phi}$,
with 
\[
\Psi\left(K\right)=\begin{pmatrix}Q & I_{r} & I_{r}\\
I_{r} & \psi^{(r)}(X) & V\\
I_{r} & V^{*} & 0
\end{pmatrix},\quad K_{\phi}=\begin{pmatrix}Q & I_{r} & I_{r}\\
I_{r} & \phi^{(r)}\left(X\right) & V\\
I_{r} & V^{*} & 0
\end{pmatrix}.
\]
Thus 
\[
\psi^{(r)}\left(X\right)=\phi^{(r)}\left(X\right)
\]
and so 
\[
\psi\left(x_{pq}\right)=\phi\left(x_{pq}\right)
\]
for every $p,q$. 

Since both maps are unital and 
\[
S=span\left\{ 1_{A},x_{pq}:1\leq p,q\leq r\right\} ,
\]
we have 
\[
\psi|_{S}=\phi.
\]
So every UCP extension of $\widehat{\phi}$ has the form $id_{M_{3r}}\otimes\psi$
for a unique UCP extension $\psi$ of $\phi$.

Conversely, if $\psi:A\longrightarrow B$ is a UCP extension of $\phi$,
then 
\[
id_{M_{3r}}\otimes\psi
\]
is UCP and its restriction to $R\left(S\right)$ is $\widehat{\phi}$.

Hence 
\[
\psi\longmapsto id_{M_{3r}}\otimes\psi
\]
is a bijection between the UCP extensions of $\phi$ and the UCP extensions
of $\widehat{\phi}$. It is affine because amplification preserves
convex combinations.

Finally, let 
\[
T=D+iK.
\]
Since $D=D^{*}$ and $K=K^{*}$, 
\[
D=\frac{T+T^{*}}{2},\quad K=\frac{T-T^{*}}{2i}.
\]
Therefore, 
\[
R\left(S\right)=span\left\{ I,T,T^{*}\right\} 
\]
and 
\[
C^{*}\left(T\right)=C^{*}\left(D,K\right)=M_{3r}\left(A\right).
\]
\end{proof}

\section{Some consequences}
\begin{cor}
\label{cor:3-1} Let $H$ be a Hilbert space and let $\phi:S\longrightarrow B\left(H\right)$
be UCP. Then $\phi$ has the unique extension property if and only
if $\widehat{\phi}$ has the unique extension property. 
\end{cor}

\begin{proof}
By \prettyref{thm:2-1}, the two extension sets are in affine bijection.
Suppose first that $\phi$ has the unique extension property, and
let $\psi:A\longrightarrow B\left(H\right)$ be its unique extension.
Then $\psi$ is a representation, and 
\[
id_{M_{3r}}\otimes\psi
\]
is the unique extension of $\widehat{\phi}$ and is again a representation.

Conversely, suppose that $\widehat{\phi}$ has the unique extension
property. By \prettyref{thm:2-1}, its unique extension has the form
\[
id_{M_{3r}}\otimes\psi
\]
for a unique $\psi\in\mathcal{E}_{A}\left(\phi,B\left(H\right)\right)$.
Since this amplification is a representation, comparison with the
$11$ corner gives 
\[
\psi\left(ab\right)=\psi\left(a\right)\psi\left(b\right),\quad a,b\in A.
\]
Thus $\psi$ is a representation, and $\phi$ has the unique extension
property. 
\end{proof}

\begin{cor}
\label{cor:3-2} Let $\pi:A\longrightarrow B\left(H\right)$ be an
irreducible representation. Then $\pi$ is a boundary representation
for $S$ if and only if $id_{M_{3r}}\otimes\pi$ is a boundary representation
for $R(S)$. Moreover, $S$ is hyperrigid in $A$ if and only if $R\left(S\right)$
is hyperrigid in $M_{3r}\left(A\right)$. 
\end{cor}

\begin{proof}
The representation $\pi$ is irreducible if and only if $id_{M_{3r}}\otimes\pi$
is irreducible. The assertion about boundary representations therefore
follows from \prettyref{cor:3-1}.

Since $S$ is finite dimensional, $A=C^{*}\left(S\right)$ is separable.
Every nondegenerate representation of $M_{3r}\left(A\right)$ is unitarily
equivalent to 
\[
id_{M_{3r}}\otimes\pi
\]
for a nondegenerate representation $\pi$ of $A$. The assertion about
hyperrigidity now follows from \prettyref{cor:3-1} and Arveson's
characterization \cite{MR2823981}. 
\end{proof}

\begin{cor}
\label{cor:3-3} Arveson's hyperrigidity conjecture fails for three
dimensional operator systems. 
\end{cor}

\begin{proof}
Let $S$ be Scherer's finite dimensional counterexample \cite{MR4956742}.
Every irreducible representation of $C^{*}\left(S\right)$ is a boundary
representation for $S$, but $S$ is not hyperrigid. By \prettyref{cor:3-2},
the same two statements hold for $R\left(S\right)$. The latter has
dimension three. 
\end{proof}

The reduction also shows that linear dimension by itself says little
about the complexity of UCP extension problems. After finite matrix
amplification, a three dimensional operator system can carry the full
extension set of an arbitrary finite dimensional operator system.
Thus low linear dimension does not imply a simple UCP extension theory.

The finite dimensional hypothesis is not essential to the reduction.
For a separable operator system, the finite matrix block containing
$S$ can be replaced by a compact diagonal block built from a countable
dense subset of $S$. The same argument then runs with compact matrix
units in place of finite matrix units. The argument is the same in
substance, so we omit it.

\bibliographystyle{amsalpha}
\bibliography{ref}

\end{document}